\documentclass[10pt]{article}

\usepackage{url}
\usepackage{amsmath}
\usepackage{amsthm}
\usepackage{amssymb}
\usepackage{latexsym}
\usepackage{pstricks}
\usepackage{graphicx, pst-plot, pst-node, pst-text, pst-tree}
\usepackage{titlefoot}
\usepackage{enumerate}
\usepackage{titlesec}
\usepackage{cancel}
\usepackage{rotating} 
\usepackage[small,it]{caption}
\usepackage{color}
\definecolor{lightgray}{rgb}{0.8, 0.8, 0.8}
\definecolor{darkgray}{rgb}{0.7, 0.7, 0.7}
\definecolor{darkblue}{rgb}{0, 0, .4}

\newtheorem{theorem}{Theorem}[section]
\newtheorem{proposition}[theorem]{Proposition}
\newtheorem{lemma}[theorem]{Lemma}
\newtheorem{definition}[theorem]{Definition}
\newtheorem{corollary}[theorem]{Corollary}

\newtheorem{remark}[theorem]{Remark}

\newtheorem{notation}[theorem]{Notation}

\newtheoremstyle{example}{\topsep}{\topsep}%
     {}
     {}
     {\bfseries}
     {.}
     {.5em}
     {\thmname{#1}\thmnumber{ #2}}
\theoremstyle{example}
\newtheorem{example}[theorem]{Example}

\newtheoremstyle{negexample}{\topsep}{\topsep}%
     {}
     {}
     {\bfseries}
     {.}
     {.5em}
     {\thmname{#1}\thmnumber{ #2}}
\theoremstyle{negexample}

\newcounter{todocounter}

\long\def\symbolfootnote[#1]#2{\begingroup%
\def\thefootnote{\fnsymbol{footnote}}\footnote[#1]{#2}\endgroup}

\makeatletter
\newcommand{\Rm}[1]{\expandafter\@slowromancap\romannumeral #1@}
\newfont{\footsc}{cmcsc10 at 8truept}
\newfont{\footbf}{cmbx10 at 8truept}
\newfont{\footrm}{cmr10 at 10truept}
\renewenvironment{abstract}%
                {
                  \begin{list}{}%
                     {\setlength{\rightmargin}{1in}%
                      \setlength{\leftmargin}{1in}}%
                   \item[]\ignorespaces\begin{small}}%
                 {\end{small}\unskip\end{list}}

\amssubj{05A17, 06A07}

\newpagestyle{main}[\small]{
        \headrule
        \sethead[\usepage][][]
        {\sc Permutation rev-tier in regards to stacks}{}{\usepage}}

\title{\sc{Passing through a stack $k$ times with reversals}}
\author{Toufik Mansour
\\[-0.25ex]
\small Department of Mathematics\\[-0.5ex]
\small University of Haifa\\[-0.5ex]
\small 3498838 Haifa, Israel\\[15pt]
Howard Skogman
\\[-0.25ex]
\small Department of Mathematics\\[-0.5ex]
\small SUNY Brockport\\[-0.5ex]
\small Brockport, New York\\[15pt]
Rebecca Smith
\\[-0.25ex]
\small Department of Mathematics\\[-0.5ex]
\small SUNY Brockport\\[-0.5ex]
\small Brockport, New York\\[-1.5ex]
}

\titleformat{\section}
        {\Large\sc}
        {\thesection.}{1em}{}

\date{}

\begin{document}
\maketitle

\pagestyle{main}

\newcommand{\s}{\mathbf{s}}
\newcommand{\m}{\mathbf{m}}
\renewcommand{\t}{\mathbf{t}}
\renewcommand{\b}{\mathbf{b}}
\newcommand{\f}{\mathbf{f}}
\newcommand{\rev}{\operatorname{rev}}
\newcommand{\dual}{\operatorname{dual}}
\newcommand{\UW}{\mathcal{UW}}
\newcommand{\R}{\mathcal{R}}
\newcommand{\D}{\mathcal{D}}
\newcommand{\E}{\mathcal{E}}
\newcommand{\inv}{\textit{inv}}
\newcommand{\Av}{\operatorname{Av}}
\newcommand{\Z}{\mathbb{Z}}

\newcommand\notsotiny{\@setfontsize\notsotiny{6}{7}}


\newcommand{\OEISlink}[1]{\href{http://oeis.org/#1}{#1}}
\newcommand{\OEISref}{\href{http://oeis.org/}{OEIS}~\cite{sloane:the-on-line-enc:}}
\newcommand{\OEIS}[1]{(Sequence \OEISlink{#1} in the \OEISref.)}

%
%
%
%
%
%
%
%

\def\sdwys #1{\xHyphenate#1$\wholeString}
\def\xHyphenate#1#2\wholeString {\if#1$%
\else\say{\ensuremath{#1}}\hspace{2pt}%
\takeTheRest#2\ofTheString
\fi}
\def\takeTheRest#1\ofTheString\fi
{\fi \xHyphenate#1\wholeString}
\def\say#1{\begin{turn}{-90}\ensuremath{#1}\end{turn}}

\newenvironment{onestack}
{
	\begin{scriptsize}
	\psset{xunit=0.033in, yunit=0.033in, linewidth=0.014in}
	\begin{pspicture}(0,3.5)(31,19)
	\psline{c-c}(0,15)(10,15)
	\psline{c-c}(13,15)(13,5)(18,5)(18,15)
	\psline{c-c}(21,15)(31,15)
	\rput[l](0,12.5){\mbox{output}}
	\rput[r](31,12.5){\mbox{input}}
}
{
	\end{pspicture}
	\end{scriptsize}
}

\newcommand{\fillstack}[3]{%
	\rput[l](0,17){\ensuremath{#1}}
	\rput[c](15.6, 9.5){\begin{sideways}{\sdwys{#2}}\end{sideways}}
	\rput[r](31,17){\ensuremath{#3}}
}

\newcommand{\stackinput}{%
	\psline[linecolor=darkgray]{c->}(24, 17)(15.5, 17)(15.5, 14.2)
}
\newcommand{\stackoutput}{%
	\psline[linecolor=darkgray]{c->}(15.5, 14.2)(15.5, 17)(10, 17)
}
\newcommand{\stackinoutput}{%
	\psline[linecolor=darkgray]{c->}(24, 17.5)(17, 17.5)(17, 13)
	\psline[linecolor=darkgray]{c->}(15,14)(15, 17.5)(10, 17.5)
}

\newcommand{\firstpass}{%
	\rput[c](16, 3.5){\text{First reverse pass.}}
}
\newcommand{\secondpass}{%
	\rput[c](16, 3.5){\text{Second reverse pass.}}
}
\newcommand{\thirdpass}{%
	\rput[c](16, 3.5){\text{Third reverse pass.}}
}

\newenvironment{threestacks}
{
	\begin{scriptsize}
	\psset{xunit=0.033in, yunit=0.033in, linewidth=0.014in}
	\begin{pspicture}(0,5)(45,19.5)
	\psline{c-c}(0,15)(10,15)(10,5)(15,5)(15,15)(20,15)(20,5)(25,5)(25,15)(30,15)(30,5)(35,5)(35,15)(45,15)
	\rput[l](-1,12.5){\mbox{output}}
	\rput[r](45,12.5){\mbox{input}}
}
{
	\end{pspicture}
	\end{scriptsize}
}

\newcommand{\fillstackseries}[5]{%
	\rput[l](-0.5,17){\ensuremath{#1}}
	\rput[c](12.6, 9.5){\begin{sideways}{\sdwys{#2}}\end{sideways}}
	\rput[c](22.6, 9.5){\begin{sideways}{\sdwys{#3}}\end{sideways}}
	\rput[c](32.6, 9.5){\begin{sideways}{\sdwys{#4}}\end{sideways}}
	\rput[r](45,17){\ensuremath{#5}}
}

\newcommand{\stackinputseries}{%
	\psline[linecolor=darkgray]{c->}(37, 17)(32.5, 17)(32.5, 14)
}
\newcommand{\stacktransfer}{%
	\psline[linecolor=darkgray]{c->}(32.5, 14)(32.5, 17)(22.5, 17)(22.5, 14)
}
\newcommand{\stacktransferandoutput}{%
	\psline[linecolor=darkgray]{c->}(32.5, 14)(32.5, 17)(23.5, 17)(23.5, 14)
	\psline[linecolor=darkgray]{c->}(22, 14)(22, 17)(13.5, 17)(13.5, 14)
	\psline[linecolor=darkgray]{c->}(12, 14)(12, 17)(8, 17)
}
\newcommand{\stacktransfertwoandoutput}{%
	\psline[linecolor=darkgray]{c->}(22, 14)(22, 17)(13.5, 17)(13.5, 14)
	\psline[linecolor=darkgray]{c->}(12, 14)(12, 17)(8, 17)
}

\newcommand{\stacktransfertwo}{%
	\psline[linecolor=darkgray]{c->}(22.5, 14)(22.5, 17)(12.5, 17)(12.5, 14)
}
\newcommand{\stackoutputseries}{%
	\psline[linecolor=darkgray]{c->}(12.5, 14)(12.5, 17)(8, 17)
}

\begin{abstract}
We consider a stack sorting algorithm where only the appropriate output values are popped from the stack and then any remaining entries in the stack are run through the stack in reverse order.  We identify the basis for the $2$-reverse pass sortable permutations and give computational results for some classes with larger maximal rev-tier.  
We also show all classes of $(t+1)$-reverse pass sortable permutations are finitely based.  Additionally, a new Entringer family consisting of maximal rev-tier permutations of length $n$ was discovered along with a bijection between this family and the collection of alternating permutations of length $n-1$.  We calculate generating functions for the number permutations of length $n$ and exact rev-tier $t$.  
\end{abstract}

\section{Introduction}~\label{intro}

We begin with the notion of permutation (or pattern) containment.

\begin{definition}  A permutation $\pi=\pi_1\pi_2\dots\pi_n\in S_n$ is said to \emph{contain} a permutation $\sigma=\sigma_1\sigma_2\ldots\sigma_k$ if there exist indices $1\le \alpha_1<\alpha_2< \ldots <\alpha_k \le n$ such that $\pi_{\alpha_i} < \pi_{\alpha_j}$ if and only if $\sigma_i <\sigma_j$.  Otherwise, we say $\pi$ \emph{avoids} $\sigma$.
\end{definition}

\begin{example}  The permutation $\pi = 4127356$ contains $231$ since the $4,7,3$ appear in the same relative order as $2,3,1$.   However, $\pi$ avoids $321$ since there is no decreasing subsequence of length three in $\pi$.
\end{example}

A stack is a last-in first-out sorting device that utilizes push and pop operations.  In Volume $1$ of \emph{The Art of Computer Programming}~\cite{knuth:the-art-of-comp:1}, Knuth showed the permutation $\pi$ can be sorted (that is, by applying push and pop operations to the sequence $\pi_1,\dots,\pi_n$ one can output the identity $1,\dots,n$) if and only if $\pi$ avoids the permutation $231$.  Subsequently Tarjan~\cite{tarjan:sorting-using-n:}, Even and Itai~\cite{even:queues-stacks-a}, Pratt~\cite{pratt:computing-permu:}, and Knuth himself in \emph{Volume 3}~\cite{knuth:the-art-of-comp:3} studied sorting machines made up of multiple stacks in series or in parallel.

Classifying the permutations that are sortable by such a machine is one of the key areas of interest in this field.  To better do so, we will use the following definitions.

\begin{definition}
A \emph{permutation class} is a downset of permutations under the containment order.  Every permutation class can be specified by the set of minimal permutations which are \emph{not} in the class called its \emph{basis}.  For a set $B$ of permutations, we denote by $\Av(B)$ the class of permutations which do not contain any element of $B$.
\end{definition}

For example, Knuth's result says that the stack-sortable permutations are precisely $\Av(231)$, that is the basis for the stack sortable permutations is $\{231\}$.  Given most naturally defined sorting machines, the set of sortable permutations forms a class.  This is because often a subpermutation of a sortable permutation can be sorted by ignoring the operations corresponding to absent entries.\footnote{An exception is West's notion of $2$-stack-sortability~\cite{west:sorting-twice-t:}, which is due to restrictions on how the machine can use its two stacks.  Namely this machine prioritizes keeping large entries from being placed above small entries.  Because of this limitation, this machine can sort $35241$, but not its subpermutation $3241$.  This restriction was extended to pop stacks by Pudwell and the third author~\cite{pudwell:two-stack-pop} where again the sortable permutations do not form a class.}

Given that the permutations sortable by a single stack are precisely $\Av(231)$, one could reasonably expect the class of sortable permutations for a network made up of multiple stacks would also be finitely based.
However, this is not the case for machines made up of $k \ge 2$ stacks in series or in parallel, shown by Murphy~\cite{murphy:restricted-perm:} and Tarjan~\cite{tarjan:sorting-using-n:}, respectively.  Moreover, the exact enumeration question is unknown; see Albert, Atkinson, and Linton~\cite{albert:permutations-ge:} for the best known bounds. For a general overview of stack sorting, we refer the reader to the survey by B\'ona~\cite{Bona:a-survey-of-sta:}.

In part because of the difficulties noted above, numerous researchers have considered weaker machines.  Atkinson, Murphy, and Ru\v{s}kuc~\cite{atkinson:sorting-with-tw:} considered sorting with two {\it increasing\/} stacks in series, i.e., two stacks whose entries must be in increasing order when read from top to bottom.\footnote{Even without this restriction, the final stack must be increasing if the sorting is to be successful.}  They characterized the permutations this machine can sort with an infinite list of forbidden patterns, and also found the enumeration of these permutations.  Interestingly, these permutations are in bijection with the $1342$-avoiding permutations previously counted by B\'ona~\cite{Bona:exact-enumerati:}.  The third author~\cite{smith:a-decreasing-st:} studied a similarly restricted machine where the first stack must have entries in decreasing order when read from top to bottom.  This permutation class of sortable permutations was shown to be $\Av(3241,3142)$ which was proven to be enumerated by the Schr\"oder numbers by Kremer~\cite{kremer:permutations-wi:,kremer:postscript:-per:}.

A different version, sorting with a stack of depth $2$ followed by a standard stack (of infinite depth), was studied by Elder~\cite{elder:permutations-ge:}.  He characterized the sortable permutations as a class with a finite basis of forbidden patterns.  Later, Elder and Goh~\cite{elder:finite-stack} showed a machine with first stack of finite depth $d \ge 3$ followed by an infinite depth stack produces a sortable class of permutations with an infinite basis.    
Yet another restriction on two stacks in series was studied by the third author and Vatter~\cite{smith:stack-and-pop-stack} by combining a pop stack with a regular stack in series.  If a pop stack is followed directly by a stack, the sortable permutations are classified by a finite basis.  If there is a queue separating the two, the cardinality of the basis for the sortable permutations is unknown, but conjectured to be finite. 

In \cite{mansour:tier}, the authors apply a sorting algorithm on a stack whereby the entries the permutation are pushed into the stack in the usual way.  An entry is popped from the stack only if it is the next needed entry for the output (the next entry of the identity permutation).
That is, allow larger entries to be placed above smaller entries, but do not allow entries to be pushed to the output prematurely.  In particular, this means that if a permutation contains the pattern $231$, then there will be entries left in the stack after all legal moves have been made.  In this case, the algorithm is repeated on the remaining entries which are returned to the input to be read in their original order.

In this paper, we begin our sorting algorithm the same way as in \cite{mansour:tier} above.  That is, prioritize outputting appropriate entries even if this restriction causes larger entries to be placed above smaller entries in the stack. 
As before, if a permutation contains the pattern $231$, there will be entries left in the stack after all legal moves have been made.  However, the new algorithm returns the remaining entries in the stack to the input in the reverse of their prior order.  

\begin{definition} The \emph{rev-tier} of a permutation is the number of times the entries in the stack must be returned to the input.  Denote the rev-tier of the permutation $\sigma$ by $t_{\rev}(\sigma)$.  

Each repetition of the stack sorting algorithm will be referred to as a \emph{reverse pass}. 
\end{definition}

When sorting a permutation $\pi$ with rev-tier $t$, this machine can be considered to be a network of $t+1$ stacks in series with a special output condition. Namely, entries of $\pi$ may only exit a stack to traverse directly to the output or if there are no more entries left to enter the stack.  This restriction has a similar flavor to the pushall stacks studied by Pierrot and Rossin~\cite{pierrot:2-stack-pushall,pierrot:2-stack} where no entry is output until all entries have entered the stacks.

\begin{example}~\label{231}
The permutation $231$ has rev-tier $t_{\rev}(231)=1$, all other elements of $S_3$ have rev-tier $0$.  Alternatively, we can say all permutations in $S_3$ are $2$-reverse-pass sortable and permutations in $S_3$ except for $231$ are $1$-reverse-pass sortable.
\end{example}

We translate Knuth's original stack sortable requirement to the following theorem.

\begin{theorem}~\label{rev-tier_0}  (Knuth) A permutation $\pi$ has positive rev-tier, that is $\pi$ cannot be sorted via single reverse pass through the stack, if and only if $\pi$ contains the pattern $231$.
\end{theorem}

\section{Classes of $(t+1)$-reverse stack sortable permutations}

To investigate the rev-tiers of permutations more generally, we will derive an explicit condition on permutations that describes their rev-tier.

\begin{definition} Let $\sigma \in S_n$ and let $i \in \{1, 2, \ldots, n - 1\}$.  Call $(i, i + 1)$ a \emph{separated pair} in $\sigma$ if there is a $k > i + 1$ 
between $i$ and $i + 1$ in $\sigma$.

A separated pair  $(i, i + 1)$ is \emph{up oriented} if $i$ precedes $i + 1$ in $\sigma$, i.e. $(i,i+1)$ is a coversion.  Otherwise  $(i, i + 1)$ is \emph{down orientated}.
\end{definition}

Equivalently one could say that $(i + 1, i)$ is a down separated pair in $\sigma$ if $i$ and $i+1$ occur as part of a $231$ pattern where $i + 1$ is the middle number and $i$ is the smallest number in the pattern. Similarly, $(i, i + 1)$ is an up separated pair if $i$ and $i + 1$ occur as the $1$ and $2$ elements in a $132$ pattern.  

These types of patterns are known as \emph{covincular} patterns. 
The formal study of these patterns was introduced by Babson and Steingr{\'{\i}}msson~\cite{babson:generalized-per:}.  For a thorough review of such pattern avoidance we refer the reader to the survey~\cite{einar:survey} by Steingr{\'{\i}}msson and book~\cite{kitaev:book} by Kitaev.

In this context, Claesson~\cite{claesson:generalized-pat} proved a result equivalent to the following proposition:

\begin{proposition} (Claesson)
\label{sep_pair_prop}
A permutation $\pi$ contains the pattern $231$ if and only if $\pi$ has a down separated pair.  Similarly, $\pi$ contains the pattern $132$ if and only if $\pi$ contains an up separated pair.
\end{proposition}

\begin{proof}
The argument for the $132$ case is nearly identical to that given for the $231$ case in \cite{mansour:tier}.  However, it is presented here for completeness and future use in this paper.   

Suppose the permutation $\sigma$ contains a $132$ pattern, say a subsequence $(a,b,c)$ with $a < c < b$.  If $a + 1= c $, then $(a, c)$ is an up separated pair. Otherwise, consider the location of $a + 1$. If $a + 1$ is to the right of
$b$, then $(a, a+1)$ is an up separated pair. If instead $a + 1$ is to the left of $b$, then there is a $132$ pattern $(a + 1, b, c)$. Iterating this process will yield an up separated pair $(i, i + 1)$ with $a \le i \le c-1$.
\end{proof}

Extending this argument allows us to characterize the rev-tier of a permutation using separated pairs.

\begin{theorem}
The rev-tier of a permutation $\pi$ under this sorting algorithm is exactly the maximum number $t$ of separated pairs $(i_i, i_1 + 1), (i_2,  i_2 +1),\ldots,(i_t, i_t+1) \in \pi$ where $i_1 < i_2 < i_3 <\ldots< i_t$ and the orientations of the pairs alternate (in this order), beginning with a down separated pair.
\end{theorem}

\begin{proof}
Let $\pi$ be a permutation.  From Theorem~\ref{rev-tier_0} and Proposition~\ref{sep_pair_prop}, $\pi$ has rev-tier $0$ if and only if $\pi$ does not contain a down separated pair. If $\pi$ contains at least one down separated pair, let $(i_1, i_1 + 1)$ be the smallest such pair. The sorting algorithm will output $12\ldots i_1$, and then the remaining entries will be processed in the reverse of their original order.  This second pass stops short of outputting the identity and instead results in a total output of $12\ldots i_2$ precisely if this new sequence contains a down separated pair $(i_2,i_2 +1)$ which was an up separated pair in $\pi$.  Each increase in rev-tier is caused by a separated pair with the opposite orientation in $\pi$ as the previous one. The theorem follows.
\end{proof}

\begin{example}  The permutation $\pi=2413$ has two alternating separated pairs, $(1,2),(2,3)$ with $(1,2)$ down separated and thus has rev-tier $2$.  We show the sorting of $\pi$ using three  reverse passes through a stack in Figure~\ref{fig_2413}.  In Figure~\ref{fig_2413_series}, the
 equivalent sorting of $\pi$ is shown with three stacks in series with the restriction of not permitting an entry $\pi_i$ to leave a stack until one of the following conditions are met:
\begin{enumerate}
\item{ $\pi_i$ is the next entry of the identity permutation for the output or}
\item{ there are no more entries to the right of the stack.}
\end{enumerate}
\end{example}

\begin{figure}[t]
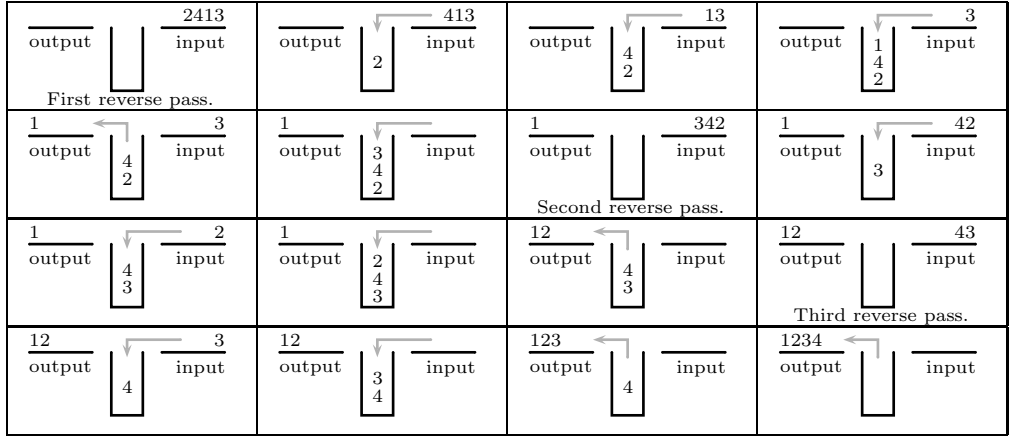

\begin{center}

\begin{tabular}{|c|c|c|c|}
\hline
\begin{onestack}
\fillstack{}{}{2413}
\firstpass
\end{onestack}
&
\begin{onestack}
\fillstack{}{2}{413}
\stackinput
\end{onestack}
&
\begin{onestack}
\fillstack{}{24}{13}
\stackinput
\end{onestack}
&
\begin{onestack}
\fillstack{}{241}{3}
\stackinput
\end{onestack}
\\\hline
\begin{onestack}
\fillstack{1}{24}{3}
\stackoutput
\end{onestack}
&
\begin{onestack}
\fillstack{1}{243}{}
\stackinput
\end{onestack}
&
\begin{onestack}
\fillstack{1}{}{342}
\secondpass
\end{onestack}
&
\begin{onestack}
\fillstack{1}{3}{42}
\stackinput
\end{onestack}
\\\hline
\begin{onestack}
\fillstack{1}{34}{2}
\stackinput
\end{onestack}
&
\begin{onestack}
\fillstack{1}{342}{}
\stackinput
\end{onestack}
&
\begin{onestack}
\fillstack{12}{34}{}
\stackoutput
\end{onestack}
&
\begin{onestack}
\fillstack{12}{}{43}
\thirdpass
\end{onestack}
\\\hline
\begin{onestack}
\fillstack{12}{4}{3}
\stackinput
\end{onestack}
&
\begin{onestack}
\fillstack{12}{43}{}
\stackinput
\end{onestack}
&
\begin{onestack}
\fillstack{123}{4}{}
\stackoutput
\end{onestack}
&
\begin{onestack}
\fillstack{1234}{}{}
\stackoutput
\end{onestack}
\\\hline
\end{tabular}

\caption{Sorting the permutation $2413$ with $k=3$ reverse-passes through a stack.}
\label{fig_2413}
\end{center}
\end{figure}

%
%

\begin{figure}[t]
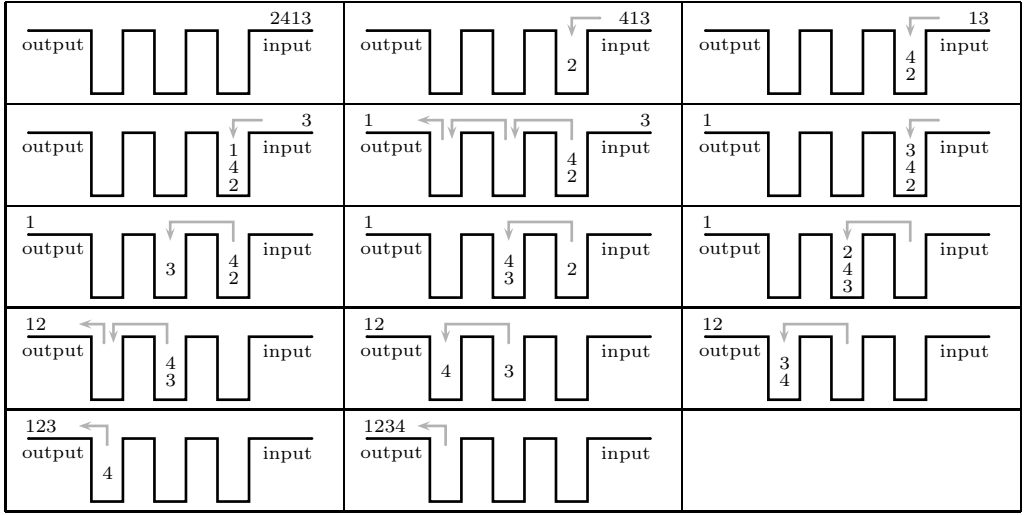

\begin{center}

\begin{tabular}{|c|c|c|}
\hline
\begin{threestacks}
\fillstackseries{}{}{}{}{2413}
\end{threestacks}
&
\begin{threestacks}
\fillstackseries{}{}{}{2}{413}
\stackinputseries
\end{threestacks}
&
\begin{threestacks}
\fillstackseries{}{}{}{24}{13}
\stackinputseries
\end{threestacks}
\\\hline
\begin{threestacks}
\fillstackseries{}{}{}{241}{3}
\stackinputseries
\end{threestacks}
&
\begin{threestacks}
\fillstackseries{1}{}{}{24}{3}
\stacktransferandoutput
\end{threestacks}
&
\begin{threestacks}
\fillstackseries{1}{}{}{243}{}
\stackinputseries
\end{threestacks}
\\\hline
\begin{threestacks}
\fillstackseries{1}{}{3}{24}{}
\stacktransfer
\end{threestacks}
&
\begin{threestacks}
\fillstackseries{1}{}{34}{2}{}
\stacktransfer
\end{threestacks}
&
\begin{threestacks}
\fillstackseries{1}{}{342}{}{}
\stacktransfer
\end{threestacks}
\\\hline
\begin{threestacks}
\fillstackseries{12}{}{34}{}{}
\stacktransfertwoandoutput
\end{threestacks}
&
\begin{threestacks}
\fillstackseries{12}{4}{3}{}{}
\stacktransfertwo
\end{threestacks}
&
\begin{threestacks}
\fillstackseries{12}{43}{}{}{}
\stacktransfertwo
\end{threestacks}
\\\hline
\begin{threestacks}
\fillstackseries{123}{4}{}{}{}
\stackoutputseries
\end{threestacks}
&
\begin{threestacks}
\fillstackseries{1234}{}{}{}{}
\stackoutputseries
\end{threestacks}
&
\\\hline
\end{tabular}

\caption{Sorting the permutation $2413$ with $k=3$ stacks in series with output restrictions.}
\label{fig_2413_series}
\end{center}
\end{figure}

We now argue that permutations up to a given rev-tier $t$ form a permutation class.

\begin{proposition}~\label{contain}
If $\sigma$ and $\tau$ are two permutations and $\sigma$ is contained in $\tau$ then $t_{\rev}(\tau) \ge$ $t_{\rev}(\sigma)$.
\end{proposition}

\begin{proof}
Suppose $\sigma$ has a maximal  alternating down/up sequence of separated pairs $(i_1, i_1 + 1), (i_2, i_2 + 1),\ldots, (i_t, i_t + 1)$ such that $i_1 < i_2 < \ldots < i_t$. 
The permutation pattern formed by the alternating sequence of separated pairs $(i_1, i_1 + 1), (i_2, i_2 + 1),\ldots, (i_t, i_t + 1)$ in $\sigma$ must be contained in $\tau$.  

Say a permutation pattern corresponding to the separated pair $(i_k,i_k+1)$ with a separator is \\ $\tau_{\alpha_k} \tau_{\beta_k} \tau_{\alpha_{k+1}}$.  Applying the argument given in the proof of Proposition~\ref{sep_pair_prop}, there exists a separated pair $(j_k,j_k+1)$ in $\tau$ (of the same orientation as that of $(i_k,i_k +1)$ in $\sigma$) where $ \tau_{\alpha_k} \leq j_k \leq \tau_{\alpha_{k+1}}-1$.  Hence $\tau$ has an alternating down/up sequence of separated pairs $(j_1, j_1 + 1), (j_2, j_2 + 1),\ldots, (j_t, j_t + 1)$ with $j_1 < j_2 < \ldots < j_t$.  Therefore $t_{\rev}(\tau) \ge$ $t_{\rev}(\sigma)$.
\end{proof}

Since the permutations of rev-tier at most $t$ are those that avoid all permutations of rev-tier $t+1$, the following corollary is a direct consequence of Proposition~\ref{contain}.

\begin{corollary}
The $(t+1)$-reverse pass sortable permutations form a permutation class for any $t \geq 0$.
\end{corollary}

\subsection{An explicit basis for $2$-reverse pass sortable permutations}

Recall the \emph{basis} of a permutation class is the minimal set of forbidden permutations any permutation in the class must avoid.

\begin{notation}
Let $B_t$ be the basis for the set of all permutations of rev-tier at most $t$.  
\end{notation}

\begin{theorem}~\label{tier_1}
The $2$-reverse pass sortable permutations are precisely $\Av(2413, 2431,23154)$.
\end{theorem}

\begin{proof}
One can verify that each of $2413, 2431,$ and $23154$ have rev-tier $2$, and that the deletion of any entry reduces the rev-tier of the permutation.  Hence $2413, 2431,23154 \in B_1$.

Now suppose $\pi$ is a basis element and thus of minimal length.  Each entry of $\pi$ must appear as part of a separated pair in the maximum length alternating down/up sequence of separated pairs in $\pi$.  Thus $(1,2)$ must be a down separated pair in $\pi$.  Similarly, since the $3$ is part of a separated pair, $\pi$ must have at least one of the following properties: 
 \begin{enumerate}
 \item{$3$ separates $(1,2)$.}
 \item{$(2,3)$ is an up separated pair.}
 \item{$(3,4)$ is an up separated pair.}
 \end{enumerate}
 
If $3$ separates $(1,2)$ and does not appear as part of a separated pair, then $(4,5)$ must be an up separated pair. 
Thus $\pi$ contains the subsequence $465$.  Also, each  of $4,5,6$ must appear either before or after the entire subsequence $231$, as otherwise deleting $3$ would result in a smaller permutation of rev-tier $2$.  However, each such configuration results in a permutation containing one of the known basis permutations $2413, 2431, 23154$.

If instead $(2,3)$ is an up separated pair, then since $(1,2)$ is a down separated pair, $\pi$ contains either the subsequence $213$ or the subsequence $231$.  Separating the first two terms of the sequence (as required) also separates the other pair and leads to $2413$ or $2431$ respectively.

Finally if $(3,4)$ is an up separated pair, then there is a $3k4$ subsequence in $\pi$ where $k>4$.  If none of these entries separate the subsequence $21$, the $\pi$ has form $2m13k4$ or $3k42m1$ where $m,k >4$.  However, the first four terms of those permutations are subpatterns $2413$ or $2431$ respectively.  Hence $\pi$ is made up of the subsequences $354$ and $21$ where at least some of the $354$ subsequence separates $21$.   The only such permutation of this form that does not contain $2413$ or $2431$ as a subpermutation is $23154$.
\end{proof}

We also were able to compute the basis for the class of $3$-reverse pass sortable permutations: 
\[
B_2= \{24153, 24513, 24531, 42513, 42531, 231564, 261453, 523164, 562413, 562431, 6723154\}.
\]

Further calculations show $B_3$ has $16$ elements of length $6$, $24$ of length $7$, $11$ of length $8$, and one of length $9$.

\subsection{Bounds on the basis elements for $(t+1)$-pass sortable permutations}

We now show there is a finite basis for each class by bounding the length of potential basis elements. 

\begin{proposition}
If $\sigma \in B_t$, then the length of $\sigma$ is at most $3(t+1)$.
\end{proposition}

\begin{proof}
If a permutation contains entries that are not part of an oriented separated pair contributing to the rev-tier, these can be eliminated without reducing the rev-tier.  A basis permutation in $B_t$ must have exactly $t+1$ such separated pairs and thus has length at most $3(t+1)$.  
\end{proof}

As an immediate consequence, we obtain the desired corollary.

\begin{corollary}  The class of $(t+1)$-reverse-pass sortable permutations has a finite basis.
\end{corollary}

We can also determine the length of the shortest elements in the basis by giving a construction of a permutation of length $n$ with maximal rev-tier.

\begin{notation} Let $\rho(n)$ represent the maximum rev-tier of any permutation of length $n$, and as before let $t_{\rev}(\sigma)$ represent the rev-tier of $\sigma$.
\end{notation}

\begin{proposition}~\label{exact_rho}
For any integer $n \ge 2$, we have $\rho(n) = n-2$.  Specifically, there are permutation(s) of length $n \geq 2$ and rev-tier $\rho(n)=n-2$ which have $(1,2), (2,3),\ldots, (n-2,n-1)$ as its maximum length alternating down/up sequence of separated pairs.
\end{proposition}

\begin{proof}
First note a permutation $\pi$ of length $n$ cannot have more separated pairs of any orientation beyond $(1,2), (2,3),\ldots, (n-2,n-1)$, so $\rho(n) \leq n-2$.  Further, one construction of a permutation of length $n$ with rev-tier $n-2$ is 
\[ 
\pi = 
\begin{cases}
(n-2)  \cdots 6 \: 4 \: 2 \: n \:1 \: 3 \:5 \cdots (n-1) & \text{if } n \text{ is even} \\
(n-1)  \cdots 6 \: 4 \: 2 \: n \:1 \: 3 \:5 \cdots (n-2) & \text{if } n \text{ is odd}. \\
\end{cases}
\]
\end{proof}

As an immediate consequence of Proposition~\ref{exact_rho}, we obtain the following corollary.

\begin{corollary}  
For any non-negative integer $t$ and any $\sigma \in B_t$, the length of $\sigma$ is greater than or equal to $t + 3$ and the bound is sharp for all $t$.
\end{corollary}

\section{Maximal rev-tier permutations are counted by the Euler numbers}

\subsection{Computational data for permutations of exact rev-tier $t$}

A simple program was written in SAGE~\cite{sagemath} to compute the rev-tier of all permutations up to length $10$. The data for the number of permutations of a given length and exact rev-tier is given in Table~\ref{table_1}.  In particular, the numbers on the top non-zero diagonal beginning at $n=3,t=1$ appear in the OEIS~\cite{OEIS} and are known as the Euler or Down/Up Numbers (A000111, A163747, A163982).

\begin{table}
\[
\begin{tabular}{|r||c|c|c|c|c|c|c|c|c|}
\hline
       & t = 0 	& t = 1 	& t = 2		& t = 3 		& t = 4 		& t = 5		& t = 6 	& t = 7		& t = 8	 		\\ \hline
n = 1  &	1	  &			&			&				&				&			&			&			&	\\ \hline
n = 2  &	2 	  & 		&			&				&				&			&			&			&	\\ \hline
n = 3  &	5	  &  1		&			&				&				&			&			& 			&	 \\ \hline
n = 4  &	14	  & 8		&	2		&				&				&			&			&			&	  \\ \hline
n = 5  &	42	  & 47		&	26		&	5			&				&			&			&			&	  \\ \hline
n = 6  & 	132	  & 248		&	228		& 96			&	16			&			&			&			&	   \\ \hline
n = 7  &	429	  & 1249	& 1702		& 1178			& 421			&	61		&			&			&	   \\ \hline
n = 8  &	1430  &	6154	& 11704		& 11840			& 6816			&	2102	&	272		&			&	  \\ \hline
n = 9  &	4862  &	30013	& 76845		& 106567		& 88020			&	43347	&	11841	&	1385	&	   \\ \hline
n = 10 &	16796 &	145764	& 490866	& 896560		& 997056		&	697644	&	302002	&	74176	& 7936	\\ \hline
\end{tabular}
\]
\caption{Number of permutations of length $n$ and exact rev-tier $t$}
\label{table_1}
\end{table}

\subsection{A new Entringer family}

We begin with some historical information on the discovery and refinement of the Euler numbers including the Entringer numbers found in ``A Survey of Alternating Permutations" by Stanley~\cite{stanley:survey-alt-perm}.


Euler was known to have studied the odd indexed terms of this sequence which have exponential generating function $\displaystyle{\sum_{n \geq 0} E_{2n+1}\frac{x^{2n+1}}{(2n+1)!} = \tan{x}}$.
The original combinatorial interpretation of Euler numbers (with any integer index) was given by Andr\'e~\cite{andre:developpements-:}.  In particular, $E_n$ is shown to count  the \emph{down/up permutations} which are \emph{alternating permutations} of length $n$ (beginning with a descent).  Note that $\displaystyle{\sum_{n \geq 0} E_{n}\frac{x^{n}}{n!} = \tan{x}+\sec{x}}$.
In 1966, Entringer~\cite{entringer:euler} published a more refined enumeration of these alternating permutations of length $n$ based on the initial term.  However, the construction of the triangle corresponding to this partitioning of the Euler numbers was published by Seidel~\cite{seidel} in 1877.

\begin{definition}
The \emph{Entringer number $E_{n,k}$} is the number of alternating permutations of $[n]$ beginning with $k$.  Thus $\displaystyle{\sum_{k=1}^{n}{E_{n,k}} = E_{n}}$.    
\end{definition}

\begin{theorem}~\label{Euler_rec} (Entringer, Seidel)

The sequence $\{E_{n,k}\}$ is defined recursively where $E_{1,1} = 1$, $E_{n,1} = 0$ when $n >1$, and
$$E_{n,k} = E_{n,k-1} + E_{n-1,n+1-k} \hspace{.5 cm} \text{ for } 1 \leq k \leq n.$$
\end{theorem}


\begin{example}  Consider the alternating permutations of length five.  Note $E_{5,5} = 5$ as it counts $$51324,51423,52314,52413,53412.$$  By exchanging the $4$ and $5$ in the previous five permutations, we have $E_{5,4}=5$ since it counts $$41325,41523,42315,42513,43512.$$  Then $E_{5,3} =4$ as it counts $31425, 31524, 32415,32514$ and finally $E_{5,2} =2$ as it counts $21435,21534.$ \\  Note $E_{5,1}=0$.  Collecting all of these permutations, we see $E_5 = 16$.
\end{example}

Since Andr\'e's discovery and Entringer's refinement, other researchers have found various combinatorial objects with the same enumeration.   In one such paper, Gelineau, Shin, and Zeng~\cite{gelineau:bijections} studied known Entringer families and also identified new Entringer families.  They created bijective proofs linking all of the twelve of the families they considered.  Others who have studied Entringer families include Poupard~\cite{poupard:two-other} and Kuznetsov, Pak, and Postinikov~\cite{kuznetsov:increasing}.

\begin{definition} Let $R_n$ be number of permutations of length $n \geq 3$ with rev-tier $n-2$.  

Define $R_{n,k}$ to be the number of permutations of length $n$ with rev-tier $n-2$ where $1$ is in position $k+1$.  

Finally, let $\R_{n,k}$ be the set of permutations of length $n$ with rev-tier $n-2$ where $1$ is in position $k+1$.
\end{definition}

\begin{remark} We note that $R_{n,n}  =R_{n,1} = R_{n,0} =0$ for all $n \geq 3$.  Thus $\displaystyle{R_n = \sum_{k=2}^{n-1} R_{n,k}}$.
\end{remark}

\begin{example}  Consider the permutations of length six with rev-tier four.  Note $R_{6,5} = 5$ as it counts $246351,246531,426351,426531,462531.$ 

Move the $1$ one position to the left to obtain $R_{6,4}=5$ counting $246315,246513,426315,426513,462513.$  

Then $R_{6,3} =4$, counting $246135,246153,426135,426153,$ 

and finally $R_{6,2} =2$ as it counts $241635,241653.$  Thus $R_6 = 16$.
\end{example}

We introduce notation for counting the number of inversions an entry of a permutation is involved in specifically as the left (larger) entry or specifically as the right (smaller) entry.  

\begin{definition}  The pair $(\pi_i, \pi_j)$ is said to be an \emph{inversion} in $\pi$ if $i<j$ and $\pi_i > \pi_j$.  
Further, let

$\inv_{L}(\pi_i) = $ the number of inversions of $\pi$ where $\pi_i$ is the first/left entry of the inversion and 

$\inv_{R}(i) = $ the number of inversions of $\sigma$ where $i$ is the second/right entry of the inversion. 
\end{definition}

The sequence $\inv_{L}(\pi_1) \ldots \inv_{L}(\pi_n)$ is called the \emph{inversion sequence} of $\pi$.  For more information on these sequences, see work by Corteel, Martinez, Savage, and Weselcouch~\cite{corteel:patterns}.

\begin{example}
The permutation $\pi = 3142$, is such that 
$\inv_{L}(\pi_1) = 2$,
$\inv_{L}(\pi_2) = 0$,
$\inv_{L}(\pi_3) = 1$, and
$\inv_{L}(\pi_4) = 0$.
Further,
$\inv_{R}(1) = 1$,
$\inv_{R}(2) = 2$,
$\inv_{R}(3) = 0$, and
$\inv_{R}(4) = 0$.
\end{example}

\begin{theorem}~\label{Euler}  The permutations of length $n \geq 3$ with rev-tier $n-2$ form an Entringer family.  Specifically, $R_{n,k} = E_{n-1,k}$ for $1 \leq k \leq n-1$ and so $R_n = E_{n-1}$ when $n \geq 3$.
\end{theorem}

\begin{proof}  Let $\E_{n,k}$ be the set of alternating permutations of length $n$ beginning with $k$.   

We construct a bijection $f: \E_{n-1,k} \rightarrow \R_{n,k}$ defining $f(\pi)$ by beginning at $\pi_1$ and proceeding left to right as follows:
\begin{enumerate}
\item Place each odd entry $2i+1$ in position $\inv_{L}(\pi_{2i+1})+2$ \emph{of the remaining positions}.
\item Place each even entry $2i$ in position $\inv_{L}(\pi_{2i})+1$  \emph{of the remaining positions}.
\item The entry $n$ is placed in the remaining spot.
\end{enumerate}

Consider the entries $j-1$ and $j$ in $f(\pi)$ where $j-1 \leq n-2$.  Specifically, consider their insertion in the permutation $f(\pi)$ by the above algorithm.

Suppose $j-1$ is even and placed in position $\inv_{L}(\pi_{j-1})+1$ of the open slots, then $j$ will be placed in position $\inv_{L}(\pi_j)+2$  of the positions that remain.  Certainly $\inv_{L}(\pi_j)+2 > \inv_{L}(\pi_{j-1})+1$ since $\pi_{j}>\pi_{j-1}$.  After $j-1$ is placed, the new position $\inv_{L}(\pi_{j-1})+1$ is the next open slot to the right of $j-1$.  Hence $j$ is placed to the right of $j-1$  in  $f(\pi)$ with at least one open position between them that must then be filled by a larger entry.  Therefore $(j-1,j)$ is an up separated pair in $f(\pi)$.

Similarly if $j-1$ is odd and placed in position $\inv_{L}(\pi_{j-1})+2$ of the remaining positions, then $j$ will be placed in position $\inv_{L}(\pi_j)+1$ of the positions that remain.  Notice $\inv_{L}(\pi_{j-1}) > \inv_{L}(\pi_j)$ since $\pi_{j-1}>\pi_{j}$ and $\pi_{j}$ is to the right of $\pi_{j-1}$.  Hence $j$ is placed to the left of $j-1$  in  $f(\pi)$ with at least one open position between them.  The only entries that remain to fill that open position are larger than $j$.  Therefore $(j-1,j)$ is a down separated pair in $f(\pi)$.

Thus $f$ takes alternating permutations of length $n-1$ to permutations of length $n$ with an alternating sequence of separated pairs $(1,2),(2,3),\ldots,(n-2,n)$ where $(1,2)$ is a down separated pair.  Note that in particular, $1$ is placed in position $\inv_{L}(\pi_1) +2 = (\pi_1-1)+2 =\pi_1+1 = k+1$.  In other words, $f$ does indeed map $\E_{n-1,k}$ into $\R_{n,k}$.

Further, this map is invertible.  
Starting with the entry $1$ in a permutation $\sigma \in \R_{n,k}$ and proceeding in order by the value of entries in $\sigma$ from $1$ to $n-1$, determine the entries of $\pi=f^{-1}(\sigma)$ as follows:
\begin{enumerate}
\item Define $\pi_{2i+1}$ to be the $\inv_{R}(2i+1)$st largest entry of those that have not already been selected.
\item Define $\pi_{2i}$ to be the $[\inv_{R}(2i) + 1]$st largest entry of those that have not already been selected.
\end{enumerate}

To see that $f^{-1}$ is indeed the inverse of $f$, consider a permutation $\sigma$ of length $n$ and rev-tier $n-2$.  Let $\pi = f^{-1}(\sigma)$.
By these definitions, $\pi_{2i+1}$ will have $\inv_{R}(2i+1) - 1$ smaller entries to its right.  That is, $\inv_{L}(\pi_{2i+1}) = \inv_{R}(2i+1) -1$.  And entry $2i+1$ of $\sigma$ is in position $\inv_{R}(2i+1)+1$ of the positions occupied by entries at least as large as $2i+1$.  Thus $\sigma$ has $2i+1$ in position $\inv_{L}(\pi_{2i+1}) +2$ among the entries at least as large as $2i+1$.
Similarly, $\pi_{2i}$ will have $[\inv_{R}(2i)+1] - 1=\inv_{R}$ smaller entries to its right.  That is, $\inv_{L}(\pi_{2i}) = \inv_{R}(2i)$.
The entry $2i$ of $\sigma$ is in position $\inv_{R}(2i)+1$ of the positions occupied entries at least as large as $2i$.   Thus $\sigma$ has $2i$ in position $\inv_{L}(\pi_{2i}) +1$ among the entries at least as large as $2i$.

 
\end{proof}

\begin{example}
Consider an alternating permutation $\pi = 21534$ which is one of the permutations counted by $\E_{5,2}$.  Then $\sigma =f(\pi)$ is such that:
\begin{flalign*}
&1  \text{ is in the }   \inv_{L}(\pi_1) + 2  =3\text{rd position of the open positions in  } \sigma. & &&
\sigma &= \underline{ \quad}\:\underline{ \quad}\:1\: \underline{ \quad}\:\underline{ \quad}\:\underline{ \quad}\\
& 2  \text{ is in the }  \inv_{L}(\pi_2) + 1  =1\text{st position of the open positions in  } \sigma. & &&
 \sigma &= 2\: \underline{ \quad}\:1\: \underline{ \quad}\:\underline{ \quad}\:\underline{ \quad}\\
& 3  \text{ is in the }   \inv_{L}(\pi_3) + 2  =4\text{th position of the open positions in  } \sigma. & &&
 \sigma &= 2\: \underline{ \quad}\:1\: \underline{ \quad}\:\underline{ \quad}\:3\\
& 4  \text{ is in the }  \inv_{L}(\pi_4) + 1  =1\text{st position of the open positions in  } \sigma. & &&
 \sigma &= 2\: 4\:1\: \underline{ \quad}\:\underline{ \quad}\:3\\
& 5  \text{ is in the }   \inv_{L}(\pi_5) + 2  =2\text{nd position of the open positions in  } \sigma. & &&
 \sigma &= 2\: 4\:1\: \underline{ \quad}\:5\:3\\
&\text{The position of } 6 \text{ in } \sigma \text{ is in the only open position in } \sigma. & &&
\sigma &= 241653 \in \R_{6,2}.
\end{flalign*}
%
\end{example}

\begin{example}
Consider a permutation $\sigma = 6247153$ with rev-tier $5$ which is one of the permutations counted by $\R_{7,4}$.  Then $\pi =f^{-1}(\sigma)$ is such that:
\begin{align*}
&\pi_1  \text{ is the }   \inv_{R}(1) =4\text{th largest value of the remaining values of } \pi, \text{ that is } \pi_1 = 4 \\
&\pi_2  \text{ is the }   \inv_{R}(2) +1 = 2\text{nd largest value of the remaining values of } \pi , \text{ that is } \pi_2 = 2 \\
&\pi_3  \text{ is the }   \inv_{R}(3) =4\text{th largest value of the remaining values of } \pi, \text{ that is } \pi_3 = 6 \\
&\pi_4  \text{ is the }   \inv_{R}(4) +1 = 2\text{nd largest value of the remaining values of } \pi , \text{ that is } \pi_4 = 3 \\
&\pi_5  \text{ is the }   \inv_{R}(5) =2\text{nd largest value of the remaining values of } \pi, \text{ that is } \pi_5 = 5 \\
&\pi_6  \text{ is the }   \inv_{R}(6) +1 = 1\text{st largest value of the remaining values of } \pi , \text{ that is } \pi_6 = 1 
\end{align*}
That is, the steps described give us $f^{-1}(6247153) = 426351 \in \E_{6,4}$.
\end{example}


\section{Generating Functions}~\label{gen-fun}

In order to determine the generating function for the number of permutations of length $n$ and exact rev-tier $t$, we refine the sets of permutations under consideration. 

\begin{definition}
A permutation $\alpha$ is \emph{up-oriented} if any maximum length increasing sequence of alternating separated pairs in $\alpha$ begins with an up separated pair. That is, suppose $ (i_1, i_1 + 1), (i_2, i_2 + 1), \ldots ,(i_p, i_p + 1)$ is a maximum length increasing sequence of separated pairs of alternating orientations in $\alpha$.  Then $\alpha$ is up-oriented precisely if $(i_1,i_1 +1)$ is an up separated pair. 
Let $M_U$ be the set of all up-oriented permutations. 

Define the corresponding notion for \emph{down-oriented} permutations and let $M_D$ denote this set. 

Finally, let $N$ denote permutations with no separated pairs. 
\end{definition}

The set of all permutations is the disjoint union of  $N, M_U,$ and  $M_D$.
Further, $N$ is a permutation class, whereas $M_U$ and $M_D$ are not permutation classes. In fact, $N = Av(132, 231)$. Let $\eta_n$ be the number of elements in $N$ 
 of length $n$, and let
\[
F^N(x) = \sum_{n\ge 1} \eta(n) x^n.
\]


The statements of Theorem~\ref{rotem_1} and Theorem~\ref{rotem_2} are equivalent to the statement and proof respectively of the last corollary of a paper by Rotem~\cite{rotem:stack}.

\begin{theorem}~\label{rotem_1} (Rotem)
With $\eta(n)$ as above, $\eta(1) = 1, \eta(2) = 2,$ and $\forall n\ge 3, \eta(n) = 2^{n - 1}$.
\end{theorem}


Next consider the refinement of the sets under consideration by the location of the smallest element of the permutation. Beginning with $N$, set 
\[
F^N(x, w) = \sum_{n\ge 1} \eta(n, k) x^n w^{k - 1}
\]
where $\eta(n, k)$ is the number of permutations in $N$ 
of length $n$ where the $1$ occurs in the $k$th position.

\begin{theorem}~\label{rotem_2} (Rotem)
With the notation as above $\eta(n, k) = \binom{n - 1}{k - 1}$.
\end{theorem}



To consider the elements in $M_U$ and $M_D$, define the operator
\[
\Theta: S = \cup_n S_n \rightarrow \Z[x,y,w], {\rm ~where~} \Theta(\alpha) = x^{n(\alpha)}y^{t_{rev}(\alpha)}w^{k(\alpha) - 1},
\]
 $n(\alpha)$ is the length of $\alpha$, $t_{rev}(\alpha)$ is the rev-tier of $\alpha$, and $k(\alpha)$ is the position of the $1$ in $\alpha$. Note that
\[
F^N(x, w) = \sum_{\alpha \in N} \Theta(\alpha).
\]
We similarly construct generating functions for all remaining permutations, namely those in $M_U$ and those in $M_D$.  Abusing notation slightly let
\[
\begin{array}{c}
\displaystyle{F(x, y, w) = \sum_{\alpha \in S} \Theta(\alpha) = \sum_{n, t, k} f(n, t, k) x^ny^tw^{k - 1},}\\ \displaystyle{M^U(x, y, w) =  \sum_{\alpha \in M_U} \Theta(\alpha) = \sum_{n, t, k} \mu_U(n, t, k) x^ny^tw^{k - 1}},\\ \displaystyle{M^D(x, y, w) =  \sum_{\alpha \in M_D} \Theta(\alpha) = \sum_{n, t, k} \mu_D(n, t, k) x^ny^tw^{k - 1}}.
\end{array}
\]

Since the set of all permutations is a disjoint union of $N, M_U$, and $M_D$, we have
\[
F(x, y, w) = F^N(x, w) + M^U(x, y, w) + M^D( x, y, w).
\]
Permutations have a down separated pair exactly if they have rev-tier at least $1$, hence
\begin{equation*}
f(n, 0, k) = \eta(n, k) + \mu_U(n, 0, k)
\end{equation*}
and 
\begin{equation*}
f(n, t, k) =  \mu_U(n, t, k) + \mu_D(n, t, k) {\rm~for~} t\ge 1.
\end{equation*}

For $i = 1, 2, \ldots, n + 1$, let $\displaystyle{\psi_i : S_{n - 1} \rightarrow S_n}$ be defined such that $\psi_i(\alpha)$ is the permutation obtained by increasing all of the values in $\alpha$ by $1$, and then inserting a $1$ in the $i$th position. Each permutation in $S_n$ is obtained by a unique $\psi_i$ applied to a unique permutation.  Now consider the action of these operators on our subsets $N, M_U$, and $M_D$. The rev-tier of the resulting permutation is recorded in the following table.
\[
\begin{array}{|c|c|c|c|}
\hline
\alpha\in {\rm Set}  & {\rm Condition} & \psi_i(\alpha)\in {\rm ~Set~} & t(\psi_i(\alpha)) \\
\hline
N & i = k, {\rm~or~} k + 1 & N & 0\\
 & i \leq k -1 & M_U & 0\\
 & i \geq k + 2 & M_D & 1\\
 \hline
M_U & i \le k + 1 & M_U & t(\alpha) \\
    & i \geq k + 2 & M_D & t(\alpha) + 2 \\
    \hline
M_D & i \leq k -1 & M_U & t(\alpha)\\
    & i \ge k & M_D & t(\alpha)\\
    \hline
\end{array}
 \]
To explain the entries in the table above, we work through each case.

Consider $\alpha \in N$. 
Then $\psi_i(\alpha)$ also has no separated pairs if and only if the resulting $1$ and $2$ are adjacent in $\psi_i(\alpha)$, that is $k(\psi(\alpha)) = k(\alpha)$ or $k(\alpha) + 1$ (where $k$ is the location of the $1$). Further $\psi_i(\alpha) \in M_U$ if and only if $(1,2)$ is an up separated pair in $\psi_i(\alpha)$, i.e. $i \leq k-1$.  Finally, $\psi_i(\alpha) \in M_D$ if and only if $(1,2)$ is a down separated pair (i.e. $i \geq k + 2$) which also  increases the rev-tier of $\psi_i(\alpha)$ to $1$.

If $\alpha$ is in $M_U$, we can create a longer increasing sequence of alternating separated pairs (ISASP) if and only if the $(1,2)$ becomes a down separated pair, i.e. $i \geq k+2$.  In this case, the longest such sequence beginning with a down separated pair increases by two, and thus the rev-tier of $\psi_i(\alpha)$ increases only in this case. Otherwise, $\psi_i(\alpha) \in M_U$ (and no change in rev-tier from $\alpha$ to $\psi_i(\alpha)$) if and only if $i \le k + 1$.

Finally if $\alpha$ is in $M_D$, then we cannot create a longer ISASP beginning with a down separated pair.  Hence the rev-tier of the image $\psi_i(\alpha)$ is the same as the rev-tier of $\alpha$. However, $\psi_i(\alpha)$ will be in $M_U$ if and only in $(1,2)$ becomes an up separated pair, i.e. if $i \leq k-1$.

Inverting the table above gives the following recurrences. 

\begin{theorem}~\label{recur_mu}
With all of the notation as above:
\begin{equation*}
\eta (n + 1, k)= \eta (n, k) + \eta (n, k - 1),
\end{equation*}
\begin{equation*}
\mu_U(n + 1, 0, k) = \sum_{i \ge k + 1} \eta(n, i) +\sum_{i \ge k - 1} \mu_U(n, 0 ,i),
\end{equation*}
and for $t \ge 1$
\begin{equation*}
 \mu_U(n + 1, t, k) = \sum_{i \ge k + 1} \mu_D(n, t, i) + \sum_{i \ge k - 1} \mu_U(n, t ,i),
\end{equation*}
\begin{equation*}
\mu_D(n + 1, 1, k) = \sum_{i \le k - 2} \eta(n, i) + \sum_{i \le k } \mu_D(n, 1 ,i),
\end{equation*}
and for $t \ge 2$
\begin{equation*}
 \mu_D(n + 1, t, k) = \sum_{i \le k - 2} \mu_U(n, t - 2, i) + \sum_{i \le k } \mu_D(n, t ,i).
\end{equation*}
\end{theorem}

Note $\mu_U(1, 0, k) = \mu_D(1, 0, k) = 0$ for all $k$, and also $\mu_D(n, 0 ,k) = 0 $ for all $n$ and $k$.
We compute a few more of these coefficients directly.
\begin{lemma}~\label{bad_1_position}
For all $n \ge 1, t \ge 0$ we have:
\begin{center}
$\mu_U(n, t, n) = \mu_U (n, t, n - 1) = 0,$\\
$\mu_D(n, t, 1) = \mu_D(n, t, 2) = 0.$\\
\end{center}
\end{lemma}

\begin{proof}
To see the conditions on $\mu_U$, suppose the $1$ appears in the penultimate or final position of $\pi$.  In these cases, consider filling in the rest of the entries of $\pi$ from smallest entry to largest.  Consider the first (if any) gap created in this process, say between an entry $k$  and the interval of entries $1,2,\ldots, k-1$.  This is the first time two consecutive entries $k-1$ and $k$ are separated by a larger entry. However, $k$ must appear before the gap that in turn precedes the smaller entries.  That is, the smallest separated pair of $\pi$, must be a down separated pair.  
Hence it is impossible for such a permutation to be in $M_U$. 

A similar argument shows that if the $1$ is in the first or second position of a permutation $\sigma$, the smallest separated pair must be up separated.  Thus such a permutation cannot be in $M_D$.
\end{proof}

A bit more can be done using elementary methods, as shown in the following lemma. 


\begin{lemma} For $t\ge 1$, 
$\mu_U(n, t, n - 2)$ is equal to the number of permutations of length less than or equal to $n - 1$, rev-tier $t$, and with $1$ in the final position. Also $\mu_D(n, t + 1, 3)$ is equal to the number of permutations of length less than or equal to $n - 1$, rev-tier $t - 1$, and a $1$ in the first position. 
That is 
\[
\mu_U(n, t, n - 2) = \sum_{m = 1}^{n - 1}f(m, t, m) = \sum_{m = 1}^{n - 1} \mu_D(m, t, m) \hspace{.5 cm} {\rm~and~} 
\]
\[\mu_D(n, t + 1, 3) = \sum_{m = 1}^{n - 1}f(m, t - 1, 1) = \sum_{m = 1}^{n - 1} \mu_U(n, t - 1, 1).
\]
\end{lemma}

\begin{proof}
We illustrate this for the $\mu_U$ case. Assume a permutation $\pi$ is of the form $**\ldots *1**$. The smallest separated pair must be up-oriented since $\pi \in M_U$. If $(1,2)$ is the smallest separated pair, then the $2$ must occur in the $n$th position. 

Now instead suppose $(k - 1, k)$ is the smallest separated pair where $k\ge 3$. Since $(1,2)$ is not a separated pair, $1$ and $2$ must be adjacent in $\pi$.  However, if $2$ is to the right of $1$,  we fall into the same situation as described in Lemma~\ref{bad_1_position} as there will not be room for both $k$ and a separator to the right of $2$.  Hence $\pi$ must be of the form $**\ldots *21**$.
 

This process continues with all entries smaller than $k$.  
Thus $\pi$ must be of the form $**\ldots *(k - 1)\ldots 21 * k$.  (Note this form is valid for $k=2$ as well.)  Now $\pi$ has only one (up) separated pair involving the entries $1,2,\ldots,k$  and $\pi$ has rev-tier $t$.  Thus entries at least as large as $k$ must form a rev-tier $t \geq 1$ permutation of length $n-k+1$, say $\pi'$ whose smallest entry is in the last position.  
Hence up to rescaling, and using Lemma~\ref{bad_1_position}, $\pi' \in M_D$.  
Therefore there are $\mu_D(n - k + 1, t, n - k + 1)$ possibilities for arranging these entries and thus for permutations in $\mu_U(n, t, n - 2)$. 
\end{proof}

Note that Lemma~\ref{rev_lemma} in the next section will show $\mu_U(n, t, n - 2) = \mu_D(n, t + 1, 3)$.

\subsection{Reversal identities}

Consider the {\it reversal map} on permutations which simply reverses the order of the elements in the permutation. That is, if $\alpha = a_1a_2\ldots a_n$, then we have $\alpha^{\rev} = a_na_{n - 1} \ldots a_2 a_1$ for all $\alpha \in S_n$. 

Let $M_U(n, t, k), M_D(n, t, k),$ and $N(n, k)$ denote the sets of permutations of length $n$, rev-tier $t$, where the $1$ is in the $k$th position with a maximal ISASP beginning with an up separated pair, beginning with a down separated pair, or having no separated pairs, respectively.
\begin{lemma}~\label{rev_lemma} 
With the notation as above we have:
\begin{equation*}
N^{\rev}(n, k) = N(n, n - k + 1) \hspace{.5 cm} \text{ and }
\end{equation*}
\begin{equation*}
 M_U^{\rev}(n, t, k) = M_D(n, t + 1, n - k + 1) \hspace{.5 cm} \text{ for all } t \geq 0.
\end{equation*}
\end{lemma}
The proof is simply to note that if a permutation $\pi$ has an ISASP with signature $UDUD\ldots$ then the reversal will have the same separated pairs but with reverse orientations.  Hence the signature of $\pi^{\rev}$ will be $DUDU\ldots$. The length of the longest ISASP overall does not change, however the rev-tier does since $\pi^{\rev}$ has a longer ISASP beginning with a down separated pair than $\pi$ does.

\begin{corollary}~\label{comp_help}
With the same notation as defined previously, for all $n,k$ we have
\begin{equation*}
\eta(n, k) = \eta(n, n - k + 1) \hspace{.5 cm} \text{ and } 
\end{equation*}
\begin{equation*}
\mu_U(n, t, k) = \mu_D(n, t + 1, n - k + 1).
\end{equation*}
\end{corollary}
The proof is simply to use the fact that the reversal map is a bijection.

Note that, if one uses the above reversal identities with the original recurrences for say $\mu_U$ (or $\mu_D$), we recover the recurrences for $\mu_D$ (or $\mu_U$ respectively) from Theorem~\ref{recur_mu}.  

We note a couple of relations between $M^U$ and $M^D$ that will be used later.
\begin{lemma}~\label{simp}
\begin{equation*}
M^U(xw, y, \tfrac{1}{w}) = \frac{w}{y} M^D(x, y, w),
\end{equation*}
\begin{equation*}
M^D(xw, y, \tfrac{1}{w}) =  wy M^U(x, y, w).
\end{equation*}
\end{lemma}

Applying Corollary~\ref{comp_help}, the proof of Lemma~\ref{simp} is simply computational.

Setting $w = 1$ we obtain the following corollary.
\begin{corollary}~\label{cor_mu}
$M^U(x, y, 1) = \frac{1}{y} M^D(x, y, 1).$
\end{corollary}
Since $F(x, y, w) = F_N(x, w) + M^U(x, y, w) + M^D(x, y, w)$, Corollary~\ref{cor_mu} gives us
\begin{corollary} 
\[
F(x, y, 1) = F_N(x, 1) + \left(1 + y\right) M^U(x, y, 1) = F_N(x, 1) + \left(1 +  \frac{1}{y}\right) M^D(x, y, 1).
\]
\end{corollary}

Next we give relations for $M^U$ and $M^D$ that allow one to calculate these generating functions explicitly.

\subsection{$\mu_U(n, 0, k)$}

Since we already have $\eta(n,k) =\binom{n - 1}{k  - 1}$, we will now use
\begin{equation*}
\mu_U(n + 1, 0, k) = \sum_{i \ge k + 1} \eta(n, i) +\sum_{i \ge k - 1} \mu_U(n, 0 ,i) \hspace{.5 cm} \text{ for all } n \geq 2
\end{equation*}
to derive the formula for $\mu_U(n, 0, k)$. Multiplying each side of the recurrence by $w^{k - 1}$  and summing on all values of $k$ we have:
\[
\sum_{k = 1}^{n + 1}\mu_U(n + 1, 0, k) w^{k - 1} = \sum_{i = 2}^{n} \eta(n, i) \frac{1 - w^{i - 1}}{1 - w}  + \sum_{i = 1}^{n} \mu_U(n, 0 ,i) \frac{1 - w^{i + 1}}{1 - w}.
\]
We now evaluate the first sum on the right given our knowledge of $\eta(n, k)$
\begin{align*}
\sum_{i = 2}^{n} \eta(n, i) \frac{1 - w^{i - 1}}{1 - w} &= \frac{1}{1 - w}
\left( \sum_{i = 2}^{n} \eta(n, i) - \sum_{i = 2}^{n} \eta(n, i) w^{i - 1} \right) \\
&= \frac{1}{1 - w}
\left( \sum_{i = 2}^{n} \binom{n - 1}{i - 1} - \sum_{i = 2}^{n} \binom{n - 1}{i - 1} w^{i - 1} \right)\\
&=\frac{1}{1 - w} \left( 2^{n - 1} - 1 - (1 + w)^{n - 1} + 1\right).
\end{align*}
Let 
\[
M_{n, 0}^U(w) =\sum_{k = 1}^n \mu_U(n, 0, k)w^{k - 1}.
\]
The above equation becomes
\begin{equation}
\label{M^U}
M_{n  + 1, 0}^U(w) = \frac{1}{1 - w}
\left( 2^{n - 1} -  (1 + w)^{n - 1} + M_{n, 0}^U(1) - w^2M_{n, 0}^U(w)\right).
\end{equation}

Now let
\[
M_0^U(x, w) = \sum_{n = 3}^{\infty} M_{n, 0}^U(w) x^n
\]
and multiply by $x^n$, and sum the previous recurrence over all positive integers $n$ to obtain:
\[
\frac{1}{x}  M_0^U(x, w) = \frac{1}{1 - w}\left( \frac{x}{1 - 2x} -  \frac{x}{ 1 - x(1 + w)} + M_0^U( x, 1) - w^2M_0^U(x, w)\right).
\]
Solving for $M_0^U(x, w)$ gives
\[
\left( 1 + \frac{xw^2}{1 - w}\right)M_0^U(x,w) = \frac{x}{1 - w}\left( \frac{x}{1 - 2x} -  \frac{x}{ 1 - x(1 + w)} + M_0^U( x, 1)\right).
\]
Employing the kernel method  so that the left hand side vanishes means setting
\[
1 + \frac{xw^2}{1 - w} = 0 {\rm~or~} 1 - w  + xw^2 = 0 {\rm~or~} w = \frac{1 \pm \sqrt{1 - 4x}}{2x}.
\]
Let $\displaystyle{\hat{w} = \frac{1 - \sqrt{1 - 4x}}{2x}}$ (the generating function for the Catalan numbers), then
\begin{equation}
M_0^U( x, 1)  = -\frac{x}{1 - 2x} +  \frac{x}{ 1 - x(1 + \hat{w})} = -\frac{x}{1 - 2x} -1 + \frac{1 - \sqrt{1 - 4x}}{2x}.
\end{equation}

By using the reversal identities, we have the rev-tier $1$ component of $M^D(x, y, 1)$, namely
\begin{equation*}
M_1^D(x, y, 1) = y M_0^U(x, 1).
\end{equation*}






\subsection{$\mu_U(n, t, k), t\ge 1$}
Notice
\begin{equation*}
 \mu_U(n + 1, t, k) = \sum_{i \ge k + 1} \mu_D(n, t, i) + \sum_{i \ge k - 1} \mu_U(n, t ,i).
\end{equation*}
Hence
\begin{align*}
\sum_{k = 1}^{n + 1} \mu_U(n + 1, t, k)w^{k - 1} &= \sum_{k = 1}^{n + 1}\sum_{i = k + 1}^n \mu_D(n, t, i)w^{k - 1} + \sum_{k = 1}^{n + 1}\sum_{i = k - 1}^n \mu_U(n, t ,i)w^{k - 1} \\
&= \sum_{i = 2}^{n}\sum_{k = 1}^{i - 1} \mu_D(n, t, i)w^{k - 1} + \sum_{i = 1}^{n}\sum_{k = 1}^{i + 1} \mu_U(n, t ,i)w^{k - 1}\\
&= \sum_{i = 2}^{n}\mu_D(n, t, i) \frac{1 - w^{i - 1}}{1 - w} + \sum_{i = 1}^{n} \mu_U(n, t ,i)\frac{1 - w^{i + 1}}{1 - w}
\end{align*}
\[
=\frac{1}{1 - w}\left( \sum_{i = 2}^{n}\mu_D(n, t, i)
 - \sum_{i = 2}^{n}\mu_D(n, t, i)w^{i - 1} + \sum_{i = 1}^{n} \mu_U(n, t ,i) - \sum_{i = 1}^{n} \mu_U(n, t ,i)w^{i + 1}\right).
\]
Let
\[
M_{n, t}^D(w) = \sum_{k = 1}^n \mu_D(n, t, k)w^{k - 1} \hspace{1.5cm} \text{ and } \hspace{1.5cm} M_{n, t}^U(w) = \sum_{k = 1}^n \mu_U(n, t, k)w^{k - 1}.
\]
Then since the $i = 1$ terms from the $\mu_D$ are zero, we have
\begin{equation}
\label{M_n^U}
M_{n + 1, t}^U(w) = \frac{1}{1 - w}\left( M_{n, t}^D(1) - M_{n, t}^D(w) + M_{n, t}^U(1) - w^2M_{n, t}^U(w)\right).
\end{equation}
Now define  
\[
M_n^{U}(y, w) = \sum_{t \ge 0}M_{n, t}^{U}(w)y^t \text{ and } ~ M_n^{D}(y, w) = \sum_{t \ge 0}M_{n, t}^{D}(w)y^t.
\]
Note that the constant term (in $y$) for $M^D$ is zero, and the constant term for $M^U$ was derived above.

Recall Equation~\ref{M^U} said $\displaystyle{M_{n + 1, 0}^U(w)=\tfrac{1}{1 - w} \left(2^{n - 1} - (1 + w)^{n - 1} + M_{n, 0}^U(1) - w^2M_{m, 0}^U(w) \right)}$. Adding this equation to Equation~\ref{M_n^U} gives us
\[
M_{n + 1}^U(y, w) = \frac{1}{1 - w} \left( M_n^D(y, 1) - M_n^D(y, w) + M_n^U(y, 1) - w^2M_{n}^U(y, w) + 2^{n - 1} - (1 + w)^{n - 1}\right).
\]
Now multiply by $x^n$, and sum on all $n\ge 1$ to get
\begin{align*}
\sum_{n = 1}^{\infty}M_{n + 1}^U(y, w)x^n &=  \frac{1}{1 - w} \left( \sum_{n = 1}^{\infty}M_n^D(y, 1)x^n - \sum_{n = 1}^{\infty}M_n^D(y, w)x^n 
+ \sum_{n = 1}^{\infty}M_n^U(y, 1)x^n - w^2\sum_{n = 1}^{\infty}M_{n}^U(y, w)x^n \right) \\
&+  \frac{1}{1 - w}\sum_{n = 1}^{\infty}(2^{n - 1} - (1 + w)^{n - 1})x^n.
\end{align*}
Let 
\[
M^{U}(x, y, w) = \sum_{n = 1}^{\infty} M_{n}^{U}(y, w)x^n \text{ and } ~ M^{D}(x, y, w) = \sum_{n = 1}^{\infty} M_{n}^{D}(y, w)x^n.
\]
Then
\begin{align*}
 M^U(x, y, w) &=  \frac{x}{1 - w} \left( M^D(x, y, 1) - M^D(x, y, w) + M^U(x, y, 1) - w^2M^U(x, y, w) \right) \\
&+ \frac{x}{1 - w}\left( \frac{x}{1 - 2x} - \frac{x}{1 - x(1 + w)} \right).
\end{align*}
Combining the $M^U(x, y, w)$  terms we obtain  
\[ 
\left(1 + \frac{xw^2}{1 - w}\right)M^U(x, y, w) =  \frac{x}{1 - w} \left( M^D(x, y, 1) - M^D(x, y, w) + M^U(x, y, 1) +  \frac{x}{1 - 2x} - \frac{x}{1 - x(1 + w)} \right).
\]
We may simplify a bit replacing $M^D(x, y, 1)$ with $yM^U(x, y, 1)$ to get
\[
\left( \frac{1 - w}{x}\right) \left(1 + \frac{xw^2}{1 - w}\right)M^U(x, y, w) =  \left( (1 + y)M^U(x, y, 1) - M^D(x, y, w) +  \frac{x}{1 - 2x} - \frac{x}{1 - x(1 + w)} \right).
\]

Note that using the reversal map (or deriving it independently), the equation we get for $M^D$ is
\begin{equation*}
\begin{split}
\left(1 - \frac{x}{1 - w}\right)M^D(x, y, w) &=
\frac{x}{1 - w}\left( w^2y^2 M^U(x, y, w) - w y^2 M^U(xw, y, 1) - w M^D(xw, y, 1) \right) \\
&+ \frac{xy}{1 - w}\left( \frac{w^2x}{1 - x(1 + w)} - \frac{w^2x}{1 - 2wx}\right).
\end{split}
\end{equation*}
Or
\begin{equation*}
\begin{split}
M^D(x, y, w) &=
\frac{x}{1 - w - x}\left( w^2y^2 M^U(x, y, w) - w y(1 + y) M^U(xw, y, 1) \right) \\
&+ \frac{xy}{1 - w - x}\left( \frac{w^2x}{1 - x(1 + w)} - \frac{w^2x}{1 - 2wx}\right).
\end{split}
\end{equation*}
Substituting this into the expression for $M^U(x,y,w)$, we have
\begin{equation}\label{eqM1}
\begin{split}
K(x,y,w)M^U(x, y, w)&= \frac{x(1 + y)}{1-w}M^U(x, y, 1) +\frac{x^2y(1+y)w}{(1-w)(1-x-w)}M^U(xw, y, 1) \\
&+\frac{x^3(2w^2x^2y-w^2xy+2w^2x+2wx^2-2wx-w-x+1)}{(1-2x)(1-x-w)(1-x-xw)(1-2xw)},
\end{split}
\end{equation}
where the kernel is defined by $K(x,y,w)=1+\frac{xw^2}{1-w}+\frac{x^2y^2w^2}{(1 - w)(1-x-w)}$. We again apply the kernel method to this functional equation. The equation $K(x,y,w)=0$ can be written as
$$(w-1)^2-(w+1)(w-1)^2x+w^2(y^2-1)x^2=0.$$
Define
$$x_0=x_0(y,w)=(1-w)\frac{1-w^2-\sqrt{(1+w^2)^2-4w^2y^2}}{2w^2(y^2-1)}=1-w+y^2w^2-y^2w^3+\cdots.$$
Clearly, $K(x_0,y,w)=0$. Also,
$$x_0=\frac{1-w}{1-y^2}\left(1-\frac{y^2}{1+w^2}C\left(\frac{w^2y^2}{(1+w^2)^2}\right)\right),$$
where $C(t)=\frac{1-\sqrt{1-4t}}{2t}$ is again the generating function for the Catalan numbers $\frac{1}{n+1}\binom{2n}{n}$. Thus, there exists a power series $w_0=w_0(x,y)$ around $x=0$ such that $x_0(y,w_0(x,y))=x$. That is, there exists a power series $w_0=w_0(x,y)$ in $x$ such that
$k(x,y,w_0(x,y))=0$.
The first terms of $w_0$ can be evaluated as
\begin{align*}
w_0&=1+y'x+y'(1+y')x^2+y'(3+5y'+2y'^2)\frac{x^3}{2}+y'(5+12y'+9y'^2+2y'^3)\frac{x^4}{2}\\
&+y'(35+112y'+125y'^2+56y'^3+8y'^4)\frac{x^5}{8}+\cdots
\end{align*}
where $y'=\sqrt{1-y^2}$. Note that when $y=0$ we have $w_0(x,0)=C(x)$.

Define $A(x,y)=\frac{x^2(w_0-1)(2w_0^2x^2y-w_0^2xy+2w_0^2x+2w_0x^2-2w_0x-w_0-x+1)}{(1+y)(1-2x)(1-x-w_0)(1-x-xw_0)(1-2xw_0)}$, and $B(x,y)=\frac{w_0}{1-x-w_0}$. Thus, Equation~\eqref{eqM1} gives
\begin{align}
\label{eqMr}
M^U(x, y, 1)=A(x,y)-xyB(x,y)M^U(xw_0(x,y), y, 1).
\end{align}

Define $M^U_j(x)=\frac{d^j}{dy^j}M^U(x,y,1)\mid_{y=0}$. Then by differentiating Equation~\eqref{eqMr} at $y=0$, we obtain
\begin{align}
\label{eqMr1}
M^U_j(x)=\frac{d^j}{dy^j}A(x,y)\mid_{y=0}-x\sum_{i=0}^{j-1}\binom{j-1}{i}\frac{d^i}{dy^i}B(x,y)\mid_{y=0}
\frac{d^{j-1-i}}{dy^{j-1-i}}(M^U(xw_0(x,y), y, 1))\mid_{y=0}.
\end{align}
Note that 
\begin{align*}
&\frac{d^m}{dy^m}(M^U(xw_0(x,y), y, 1))\\
&=\frac{d^{m-1}}{dy^{m-1}}\left(x\frac{d}{dy}w_0(x,y)\frac{d}{ds_1}M^U(s_1, y, 1)\mid_{s_1=xw(x,y)}+\frac{d}{ds_2}M^U(xw(x,y),s_2,1)\mid_{s_2=y}\right).
\end{align*}
Thus, by induction on $m$, the expression $\frac{d^m}{dy^m}(M^U(xw_0(x,y), y, 1))$ can be written in terms of $M^U_j(x)$ and derivatives of $M^U(x,0,1)=M^U_0(x)$. Hence, Equation~\eqref{eqMr1} defines a procedure for finding an explicit formula for the generating function $\frac{1}{j!}M^U_j(x)$ which is the coefficient of $y^j$ in the generating function $M^U(x,y,1)$. For instance, we apply our procedure for $j=0,1,2$. 

{\bf Case $M^U_0(x)$}: When $y=0$ (here $w_0(x,0)=C(x)$), after simplification by using the fact that $C(x)=1+xC^2(x)$, we see that  
$M^U_0(x)=A(x,0)=C(x)-\frac{1-x}{1-2x}$, which leads to the following result.
\begin{corollary}
We have 
$$M^U_0(x)=C(x)-\frac{1-x}{1-2x} = x^3 + 6x^4 + 26x^5 + 100x^6 + 365x^7 + \cdots.$$
\end{corollary}

{\bf Case $M^U_1(x)$}: By \eqref{eqMr1}, and $w_0(x,0)=C(x)$, we have
\begin{align*}
M^U_1(x)&=\frac{d}{dy}A(x,y)\mid_{y=0}-xA(xC(x),0)B(x,0)\\
&=\frac{xC(x)(\frac{1-xC(x)}{1-2xC(x)}-C(xC(x)))}{1-x-C(x)}
+\frac{x^4C^6(x)(1-x+x(2x-3)C(x))}{(1-2x)(1-x-C(x))(1-2xC(x))}.
\end{align*}

\begin{corollary}
Let $t_0=C(x)$, and $t_1=C(xC(x))$. We have
\begin{align*}
M^U_1(x)&=\frac{xt_0(\frac{1-xt_0}{1-2xt_0}-t_1)}{1-x-t_0}
+\frac{x^4t_0^6(x)(1-x+x(2x-3)t_0)}{(1-2x)(1-x-t_0)(1-2xt_0)}\\
&=2x^4+21x^5+148x^6+884x^7+4852x^8+25407x^9+129480x^{10}+649576x^{11}+\cdots.
\end{align*}
\end{corollary}

{\bf Case $M^U_2(x)$}: Similarly, we can obtain the next case.
\begin{corollary}
Let $t_0=C(x)$, $t_1=C(xC(x))$, and $t_2=C(xC(x)C(xC(x)))$. We have
\begin{align*}
M^U_2(x)&=\frac{2x^2\left(t_1t_0^2(2x-1)(t_0-2)(xt_0+x-1)^2(xt_0t_1+xt_0-1)t_2+\sum_{i=0}^4H_i(x)x^i\right)}{(xt_0t_1+xt_0-1)(xt_0+t_1-1)(t_0^2+xt_0-3t_0-2x+2)(xt_0+x-1)^2(2x-1)}\\
&=10x^5+160x^6+1636x^7+13704x^8+102876x^9+722772x^{10}+4867904x^{11}+\cdots,
\end{align*}
where 
\begin{align*}
H_0(x)&=(2-t_0)(t_0^2+2t_0t_1-2t_0-t_1+1),\\
H_1(x)&=-2t_0^4t_1^2+3t_0^4t_1+6t_0^3t_1^2+3t_0^4-5t_0^2t_1^2-12t_0^3-9t_0^2t_1+2t_0t_1^2+11t_0^2-7t_0t_1+3t_0+4t_1-4,\\
H_2(x)&=2t_0^5t_1^2-6t_0^5t_1-2t_0^4t_1^2-3t_0^5-7t_0^3t_1^2+12t_0^4+14t_0^3t_1+7t_0^2t_1^2-3t_0^3+18t_0^2t_1-4t_0t_1^2-18t_0^2\\
&+3t_0t_1+5t_0-2t_1+2,\\
H_3(x)&=t_0\bigl(3t_0^5t_1-2t_0^4t_1^2+t_0^5+6t_0^4t_1+3t_0^3t_1^2-4t_0^4-14t_0^3t_1+4t_0^2t_1^2-6t_0^3-15t_0^2t_1-3t_0t_1^2+15t_0^2\\
&-12t_0t_1+2t_1^2+6t_0-4\bigr),\\
H_4(x)&=-t_0^2(4t_0^4t_1-3t_0^3t_1-3t_0^3-10t_0^2t_1+2t_0^2-t_0t_1+7t_0-2t_1-2).
\end{align*}
\end{corollary}

Recall that the generating function $F(x, y, w) = F^N(x, w) + (1 + y)M^U(x, y, w)$, and hence for any $t\ge 1$ the number of permutations of length $n$, and rev-tier $t$ is the $y^t$ term of $(1 + y)M^U(x, y, 1)$. Thus 
\[
\sum_{k = 1}^n f(n, t, k) = \frac{M_t^U(x)}{t!} + \frac{M_{t - 1}^U(x)}{(t - 1)!} \hspace{.5 cm} \text{ for all } t \geq 1.
\]
For example, we obtain the following corollary:
\begin{corollary}
The generating function for  the number of permutations of length $n$, and rev-tier $2$ is
\[ 
\frac{M_2^U(x)}{2!} + M_{1}^U = 2x^4 + 26x^5 + 228x^6 + 1702x^7 + \cdots.
\]
\end{corollary}

\section{A new unbalanced Wilf equivalence}

We conclude with a summary of unbalanced Wilf equivalences.  Two permutation classes are said to be \emph{Wilf equivalent} if their enumeration is the same.  The terminology \emph{unbalanced Wilf equivalence} refers to two Wilf equivalent permutation classes whose bases have unequal cardinalities.  The first unbalanced Wilf equivalence was found by Atkinson, Murphy, and Ru\v{s}kuc~\cite{atkinson:sorting-with-tw:} when combining their result with that of B\'ona~\cite{Bona:exact-enumerati:} in 1997 as mentioned in Section~\ref{intro}.  In 2016, Egge~\cite{egge:talk-unbalanced} conjectured there were also unbalanced Wilf equivalences between two permutation classes whose bases are both finite.  Burstein and Pantone~\cite{burstein:unbalanced-wilf} and Bloom and Burstein~\cite{bloom:egge-triples} proved the first examples of this conjecture.  

Our data (some of which is shown in the second column of Table~\ref{table_2}), suggested the enumeration of the class of $2$-reverse pass sortable permutations is given by sequence A165543 in OEIS~\cite{OEIS}.  The sequence was first proven by Callan~\cite{callan:permutations} to enumerate the class $\Av(4321, 4213)$  with a  more intuitive bijection (as requested by Callan) later given by Bloom and Vatter~\cite{bloom:two-vignettes}.  A conjecture of a new unbalanced Wilf equivalence based on this work and Theorem~\ref{tier_1} was presented at Permutation Patterns 2018.  
Soon after,  Bean~\cite{temp} proved it with the aid of a computer program he wrote.

\begin{theorem}
The permutation classes $\Av(4321, 4213)$ and $\Av(2413, 2431, 23154)$ are Wilf equivalent.  Specifically, both classes of permutations are enumerated by the generating function
$$\frac{1}{1-xC(xC(x))}$$
where $C(x)$ is the generating function for the Catalan numbers, sequence  A000108 in OEIS~\cite{OEIS}.
\end{theorem}

This theorem can also be proved by summing the generating functions found in Section~\ref{gen-fun} for $t=0,1$.
\begin{center}
\begin{table}
\[
\begin{tabular}{|r||c|c|c|c|c|c|c|c|c|}
\hline
       & t = 0 	& t $\le$ 1 	& t $\le$ 2		& t $\le$ 3 		& t $\le$ 4 		& t $\le$ 5		& t $\le$ 6 & t $\le$ 7	& t $\le$ 8		\\ \hline
n = 1  &	1	  &	1	&	1	&	1	&	1	&	1		&	1		&		1	&	1\\ \hline
n = 2  &	2 	  & 	2	&	2	&	2	&	2	&	2		&	2		&		2	&	2\\ \hline
n = 3  &	5	  &  6		&	6	&	6	&	6	&	6		&	6		& 		6	&	6 \\ \hline
n = 4  &	14	  & 22	&	24	&	24	&	24	&	24		&	24		&		24	&	24  \\ \hline
n = 5  &	42	  & 89	&	115	&	120	&	120	&	120		&	120		&		120	&	 120 \\ \hline
n = 6  & 	132	  & 380	&	608	& 704	&	720	&	720		&	720		&		720	&	720   \\ \hline
n = 7  &	429	  & 1678	& 3380	& 4558	& 4979	&	5040		&	5040		&		5040	&	5040   \\ \hline
n = 8  &	1430  &	7584	& 19288	& 31128	& 37946	&	40048	&	40320	&		40320&	40320  \\ \hline
n = 9  &	4862  &	34875& 111720	& 218287	& 306307	&	349654	&	361495	&	362880	&	362880   \\ \hline
n = 10 &	16796 &	162560& 653426& 1549986& 2547042&	3244686	&	3546688	& 3620864	& 3628800	\\ \hline
\end{tabular}	
\]
\caption{Number of permutations of length $n$ and rev-tier at most $t$}
\label{table_2}
\end{table}
\end{center}


\bibliographystyle{acm}
\bibliography{refs_tier}

\end{document}